\documentclass[12pt,twoside,english,reqno,a4paper,final]{amsart}
\usepackage[utf8]{inputenc}
\usepackage[english]{babel}
\usepackage{cancel}
\usepackage[lined,boxed]{algorithm2e}
\usepackage{float} %needed to fix the position of a figure

\usepackage{fullpage}

\usepackage{listings,graphicx,amsmath,varioref,amscd,amssymb,color,xcolor,bm,amsthm,amsfonts,graphics,lineno,float,comment}
\usepackage{mathrsfs} % mathscr F

\newcommand{\de}{\partial}

\theoremstyle{plain}
\newtheorem{theorem}{Theorem}[section]

\newtheorem{lemma}[theorem]{Lemma}

\theoremstyle{definition}
\newtheorem{defin}[theorem]{Definition}

\newtheorem{remark}[theorem]{Remark}
\newtheorem{example}{Example}

\theoremstyle{remark}

\numberwithin{equation}{section}
\author[F. Della Pietra]{Francesco Della Pietra}
\author[F. Oliva]{Francescantonio Oliva}
\author[S. Segura]{Sergio Segura de Le\'on}

\address{Francesco Della Pietra - Dipartimento di Matematica e Applicazioni ``Renato Caccioppoli'', Universit\`a degli studi di Napoli Federico II, Via Cintia, Monte S. Angelo, 80126 Napoli, Italy}
\email{f.dellapietra@unina.it}

\address{Francescantonio Oliva - Dipartimento di Scienze di Base e Applicate per l'Ingegneria, Sapienza Università di Roma, Via Antonio Scarpa 16, 00161 Roma, Italy}
\email{francescantonio.oliva@uniroma1.it}
\address{Sergio Segura de Le\'on - Departament d'An\`alisi Matem\`atica, Universitat de Val\`encia, Dr. Moliner 50,
	46100 Burjassot, Val\`encia, Spain}
\email{sergio.segura@uv.es}


\begin{document}

\title{On a nonlinear Robin problem with an absorption term on the boundary and $L^1$ data}

\maketitle

\begin{abstract}
	We deal with existence and uniqueness of nonnegative solutions to
	\begin{equation*}
		\label{pbabstract}
		\begin{cases}
			\displaystyle -\Delta u = f(x) & \text{ in }\Omega,\\
			\displaystyle \frac{\partial u}{\partial \nu} + \lambda(x) u = \frac{g(x)}{u^\eta} & \text{ on } \partial\Omega,
		\end{cases}
	\end{equation*}
	where $\eta\ge 0$ and $f,\lambda$ and $g$ are nonnegative integrable functions. The set $\Omega\subset\mathbb{R}^N (N> 2)$ is open and bounded  with smooth boundary and $\nu$ denotes its unit outward normal vector.
	  \\
	  More generally, we handle equations driven by monotone operators of $p$-Laplacian type jointly with nonlinear boundary conditions. We prove existence of an entropy solution and check that this solution is unique under natural assumptions. Among other features, we study the regularizing effect given to the solution by both the absorption and the nonlinear boundary term.
\end{abstract}

\begin{center}
\begin{minipage}{.65\textwidth}
\tableofcontents
\end{minipage}

\end{center}

\section{Introduction}

In this paper we analyze existence and uniqueness of nonnegative solutions to the following model problem
\begin{equation}
	\label{pbintro}
	\begin{cases}
		\displaystyle -\Delta u = f(x) & \text{ in }\Omega,\\
		\displaystyle \frac{\partial u}{\partial \nu} + \lambda(x) u = \frac{g(x)}{u^\eta} & \text{ on } \partial\Omega,
	\end{cases}
\end{equation}
where $\eta\ge 0$ and $f,\lambda$ and $g$ are nonnegative functions which can be even merely integrable. Here $\Omega$ denotes an open bounded subset of $\mathbb{R}^N$ $(N>2)$ with smooth boundary while $\nu$ denotes its unit outward normal vector.

The main interest in problem \eqref{pbintro} relies on the boundary equation which contains a blowing up term providing a nonlinear Robin boundary condition.

\medskip

Up to our knowledge, problems as \eqref{pbintro} are essentially new to the literature. Singular problems are extremely studied when the singularity is, let say, volumetric and a Dirichlet boundary condition is imposed; this provides to the problem a singular feature.

Let us just cite some historical papers \cite{crt,lm} and also more recent papers \cite{bo,ddo,op} which investigate different aspects of this problem using various techniques.

\medskip

The literature concerning nonlinear boundary conditions is more limited and mainly focused on $u_\nu + F(u)=0$ with $F$ nondecreasing and finite at the origin (\cite{aimt,ams,go,pri}). Apart from these, in \cite{mt} the authors deal with a problem involving a function $F$ blowing up at the origin in a concrete model case. Here it is proven existence and nonnexistence results in presence of subcritical powers in both the interior and the boundary equation. It is also worth to mention papers \cite{gmm,gmp} where the authors consider homogeneous Robin boundary conditions jointly with singular terms even dependant on the gradient of the solution itself when $\lambda$ is a positive constant. In particular they show existence of solutions using variational methods and a sub and super-solution technique. Let also mention that, if $g \equiv 0$ and $\lambda$ is a positive constant, recent results can be found in \cite{ant}. Moreover, if $\eta=0$ and the principal operator is the $p$-Laplacian, one can also refer for instance to \cite{dos} where, among other things, the asymptotic behaviour of the solutions as $p\to 1^+$ is studied.

\medskip

In this paper we first show existence of a weak solution to \eqref{pbintro} when the data are regular enough (see Theorem \ref{teoexreg} below). Here, through a suitable regularization process on the data involved, we exploit a comparison argument as well as the classical Hopf Lemma in order to deduce that the approximation sequence is bounded from below on $\partial\Omega$. Roughly speaking, this gives to the problem a non-singular feature which allows to easily pass to the limit the approximation sequence. Let explicitly stress that we strongly need $\lambda$ to be bounded and not null in the previous argument; in particular this permits to deduce that the sub-solution to the approximation sequence given by \eqref{pbsottosol} is actually positive on the boundary of the domain.

Let also mention that the regularizing effect given by the singular boundary term is expressed by condition \eqref{erresecreg} below. If $\eta>0$ we obtain finite energy solutions for a larger class of data; for instance, if $\eta\ge 1$ we need $g$ to be just an integrable function to have the solution $u\in W^{1,p}(\Omega)$. This effect is due to the degeneration at infinity of the nonlinear boundary term.

\medskip

In the second part of the paper we deal with a generalization of problem \eqref{pbintro} given by

\begin{equation}
	\label{pbmaingeneralintro}
	\begin{cases}
		\displaystyle -\operatorname{div}(a(x,\nabla u)) = f & \text{ in }\Omega,\\
		u\ge 0 & \text{ in }\Omega,\\
		\displaystyle a(x,\nabla u)\cdot \nu + \lambda \sigma(u) = h(u)g & \text{ on } \partial\Omega,
	\end{cases}
\end{equation}

where $a$, $\sigma$ and $h$ are suitable generalizations of the functions involved in \eqref{pbintro} satisfying assumptions \eqref{cara1},\eqref{cara2}, \eqref{cara3}, \eqref{k}, \eqref{klim} and \eqref{sigma}  below. Finally $\lambda$ and $g$ are merely nonnegative integrable functions on $\partial \Omega$ as well as $0\le f\in L^1(\Omega)$.

As it is clear problem \eqref{pbmaingeneralintro} is non-variational and here the approximation process is strongly needed to show the existence of entropy solutions (see Definition \ref{def_ent} below).

In this case we need a totally different strategy with respect to the one of Theorem \ref{teoexreg}: indeed we can not show that the approximation sequence is bounded from below on $\partial\Omega$ since, among other things, $\lambda$ is actually unbounded. Here we take advantage of suitable test functions to control the nonlinear (and possibly singular) boundary term.

Let finally stress that Theorem \ref{teouniqueent} below shows that the entropy solution to \eqref{pbmaingeneralintro} is unique under natural monotonicity assumptions on the involved functions.

\medskip

The plan of the paper is the following: in Section \ref{sec_reg} we deal with existence of a weak solution to \eqref{pbintro}. In Section \ref{sec_entropy} we prove existence and uniqueness of entropy solutions to a generalization of \eqref{pbmaingeneralintro} and we finally prove the uniqueness theorem.

\subsection{Notation and preliminaries}
\label{not}
For the entire paper $\Omega$ is an open bounded set of $\mathbb{R}^N$ ($N\ge 2$) with regular boundary.
\\
For a given function $v$ we denote by $v^+=\max(v,0)$ and by $v^-= -\min (v,0)$. Moreover $\chi_{E}$ denotes the characteristic function of a set $E$. For a fixed $k>0$, we define the truncation function $T_{k}:\mathbb{R}\to\mathbb{R}$ %and $G_{k}:\mathbb{R}\to\mathbb{R}$%
as
\begin{align*}
	T_k(s):=&\max (-k,\min (s,k)).\\
%	G_k(s):=&(|s|-k)^+ \operatorname{sign}(s).
\end{align*}
We will also use the functions
\begin{align}\label{Vdelta}
	\displaystyle
	V_{\delta}(s):=
	\begin{cases}
		1 \ \ &s\le \delta, \\
		\displaystyle\frac{2\delta-s}{\delta} \ \ &\delta <s< 2\delta, \\
		0 \ \ &s\ge 2\delta,
	\end{cases}
\end{align}
and
\begin{align}\label{phit}
	\displaystyle
	\phi_{t,\varepsilon}(s):=
	\begin{cases}
		0 \ \ &s\le t, \\
		\displaystyle\frac{s-t}{\varepsilon} \ \ &t <s< t+\varepsilon, \\
		1 \ \ &s\ge t+\varepsilon.
	\end{cases}
\end{align}

Fixed a nonnegative $\lambda\in L^{\frac{N-1}{p-1}}(\Omega)$ (not identically null), we consider in $W^{1,p}(\Omega)$ the norm defined by
\begin{equation*}\label{norma}
	\| v\|^p_{\lambda,p}=\int_\Omega |\nabla v|^p+\int_{\partial\Omega}\lambda |v|^p d \mathcal{H}^{N-1}\qquad v\in W^{1,p}(\Omega)\,.
\end{equation*}
This norm turns out to be equivalent in $W^{1,p}(\Omega)$ to the usual norm (see \cite[Section 2.7]{N}). As a consequence, classical embeddings that hold for $W^{1,p}(\Omega)$ can be translated to this norm.

Let also recall the following well known trace inequality (see \cite[Theorem $4.2$]{N}). There exists $C>0$ such that:
\begin{equation}\label{trace}
	\|v\|_{L^{\frac{(N-1)p}{N-p}}(\partial\Omega)} \le C\|v\|_{W^{1,p}(\Omega)}, \ \forall v \in W^{1,p}(\Omega).
\end{equation}
It is worth mentioning that the previous immersion is also compact in $L^q(\partial\Omega)$ if $q<\frac{(N-1)p}{N-p}$ (see \cite[Theorem $6.1$]{N}).

For any $0<r<\infty$, by $M^r(\Omega)$ we denote the usual Marcinkiewicz (or weak Lebesgue) space of index $r$, which is the space of functions $f$
such that $|\{|f|>t\}|\le C t^{-r}$, for any $t>0$. Let only recall that, if $|\Omega|<\infty$, $L^{r}(\Omega)\subset M^r(\Omega)\subset L^{r-\varepsilon}(\Omega)$, for any $\varepsilon>0$. For an overview to these spaces we refer to \cite{hunt}.

\medskip

If no otherwise specified, we will denote by $C$ several positive constants whose value may change from line to line and, sometimes, on the same line. These values will only depend on the data but they will never depend on the indexes of the sequences we will introduce.
\\
Finally we underline that, if no ambiguity occurs,  we will often use the following notation for the Lebesgue integral of a function $f$
$$
 \int_\Omega f:=\int_\Omega f(x)\ dx.
$$

\section{The case with regular data}
\label{sec_reg}
In this section, under the assumption of $\Omega$ bounded open set with $C^{1}$ boundary, we prove existence of solution to the following model problem
\begin{equation}
	\label{pbmainreg}
	\begin{cases}
		\displaystyle -\Delta u = f & \text{ in }\Omega,\\
		\displaystyle u\ge 0 & \text{ in }\Omega,\\
		\displaystyle \frac{\partial u}{\partial \nu} + \lambda u = \frac{g}{u^\eta} & \text{ on } \partial\Omega,
	\end{cases}
\end{equation}
where $\eta\ge 0$ and $f\in L^{\frac{2N}{N+2}}(\Omega)$, $\lambda\in L^\infty(\partial\Omega)$ (not identically null)
 and $g\in L^r(\partial\Omega)$ are nonnegative and
\begin{equation}\label{erresecreg}
	r = \max\left(\frac{2(N-1)}{N+\eta (N-2)},1\right).
\end{equation}
The main interesting fact in this section is that, under the above assumptions and through classical tools, the solution is far away from zero on $\partial \Omega$.  Roughly speaking, this means that problem \eqref{pbmainreg} is non-singular.
\\
Let us firstly precise what we mean by a weak solution.
\begin{defin}\label{weakdefreg}
	A function $u\in H^1(\Omega)$ is a weak solution to \eqref{pbmainreg} if $gu^{-\eta} \in L^1(\partial\Omega)$  and if it satisfies
	\begin{equation}\label{weakdefregformulation}
		\int_\Omega \nabla u\cdot \nabla \varphi +  \int_{\partial\Omega} \lambda u\varphi d \mathcal{H}^{N-1}  = \int_\Omega f\varphi + \int_{\partial\Omega} \frac{g\varphi}{u^\eta} d \mathcal{H}^{N-1},
	\end{equation}
	for all $\varphi \in H^1(\Omega)\cap L^\infty(\partial\Omega)$.
\end{defin}

Let us state the existence result for this section.
\begin{theorem}\label{teoexreg}
	Let $0\le f\in L^{\frac{2N}{N+2}}(\Omega)$, let $0 \le \lambda \in L^{\infty}(\partial\Omega)$ be not identically null and let $0\le g\in L^r(\partial\Omega)$ with $r$ satisfying \eqref{erresecreg}. Then there exists a weak solution to \eqref{pbmainreg}.
\end{theorem}

\begin{remark}
Let us stress that the previous existence result concerns nonnegative solutions. Anyway simple basic examples show that, in general, 	changing sign solutions exist as shown in Example \ref{ejem} below. Roughly speaking, here we are formally dealing with existence of solutions to
	\begin{equation*}
		\begin{cases}
			\displaystyle -\Delta u = f & \text{ in }\Omega,\\
			\displaystyle \frac{\partial u}{\partial \nu} + \lambda u = \frac{g}{|u|^\eta} & \text{ on } \partial\Omega,
		\end{cases}
	\end{equation*}
which are nonnegative. Let us also underline that the study of problems as in \eqref{pbmainreg} where $f, g$ are not necessarily positive is the object of a forthcoming paper. Obviously, in this case, nonnegative solutions are not always expected to exist.
\end{remark}
\begin{example}\label{ejem}
	Let $B_1(0)$ be the unit ball in $\mathbb R^2$ and let us consider the following problem
	\begin{equation}\label{ejemeq}
		\left\{\begin{array}{ll}
			-\Delta u=0&\hbox{in }B_1(0)\\[3mm]
			\displaystyle \frac{\partial u}{\partial \nu}+\lambda u=\frac{g}{u}&\hbox{on }\partial B_1(0)
		\end{array}\right.
	\end{equation}
	
	In what follows, we use polar coordinates $0\le r\le 1$ and $-\pi<\theta\le \pi$. If we fix the nonnegative functions
 	\[\lambda(\theta)=\frac1{|\theta|^\alpha}\,,\]
 	with $0\le \alpha<1$ and
	 \[g(\theta)=\sin^2\theta\left(1+\frac1{|\theta|^\alpha}\right),\]
	which give $\lambda\in L^1(\partial B_1(0))$ while $g\in L^\infty(\partial B_1(0))$
	Then it is simple to convince that $u(r,\theta)=r\sin \theta$ is a solution to \eqref{ejemeq}.
	
	Let observe that $u$ vanishes on the boundary at $\theta=0$ and  $\theta=\pi$. At $\theta=0$, function $\lambda$ exhibits a singularity. However, at  $\theta=\pi$, the weight $\lambda$ is bounded. Moreover, when $\alpha=0$, $\lambda$ is bounded but both zeros remain.
	
\end{example}

\subsection{Approximation scheme and proof of the existence result}

In order to prove the above theorem, we work by approximation through the following problems
\begin{equation}
	\label{pbapprox}
	\begin{cases}
		\displaystyle -\Delta u_n = f_n & \text{ in }\Omega,\\
		\displaystyle u_n \ge 0 & \text{ in }\Omega,\\
		\displaystyle \frac{\partial u_n}{\partial \nu} + \lambda u_n = \frac{g_n}{\left(|u_n|+\frac{1}{n}\right)^\eta} & \text{ on } \partial\Omega,
	\end{cases}
\end{equation}
where $f_n:= T_n(f)$ and $g_n:=T_n(g)$. We first show the existence of a weak solution to \eqref{pbapprox}, namely a function $u_n\in H^1(\Omega)$ satisfying
\begin{equation}\label{weakfor}
	\int_\Omega \nabla u_n\cdot \nabla \varphi +  \int_{\partial\Omega} \lambda u_n\varphi d \mathcal{H}^{N-1}  = \int_\Omega f_n\varphi + \int_{\partial\Omega} \frac{g_n\varphi}{\left(|u_n|+\frac{1}{n}\right)^\eta} d \mathcal{H}^{N-1},
\end{equation}
for all $\varphi \in H^1(\Omega)$.
\begin{lemma}\label{lem_ex}
	Let $0\le f \in L^1(\Omega)$, let $0\le \lambda\in L^\infty(\partial\Omega)$ not identically null and let $0\le g\in L^1(\partial\Omega)$. Then there exists a nonnegative weak solution $u_n$ to \eqref{pbapprox}.
\end{lemma}
\begin{proof}
	In order to show the existence of a solution to \eqref{pbapprox} let us consider
\begin{equation}
	\label{pbapproxschauder}
	\begin{cases}
		\displaystyle -\Delta w = f_n & \text{ in }\Omega,\\
		\displaystyle \frac{\partial w}{\partial \nu} + \lambda w = \frac{g_n}{\left(|v|+\frac{1}{n}\right)^\eta} & \text{ on } \partial\Omega,
	\end{cases}
\end{equation}	
	where $v\in L^2(\partial \Omega)$. The existence of a solution $w\in H^1(\Omega)$ to \eqref{pbapproxschauder} follows, for example, from the classical results contained in \cite{ll}; moreover it is simply to deduce that $w$ is actually nonnegative and, from a classical argument by Stampacchia, it is also bounded. In order to deduce the existence of a solution $u_n$ to \eqref{pbapprox} we aim to show that the application $T: L^2(\partial\Omega)\mapsto L^2(\partial\Omega)$ such that $T(v)=w\big|_{\partial\Omega}$ admits a fixed point. Hence it will be sufficient to show that $T$ is invariant, compact and continuous to  apply the Schauder fixed point Theorem in order to conclude the proof.
	
	\medskip
	
	We start by proving that $T$ is invariant; to this end let us take $w$ as a test function in the weak formulation in \eqref{pbapproxschauder} deducing that (recall $f_n,g_n\le n$)
	\begin{equation*}\label{exapprox1}
		\int_\Omega |\nabla w|^2 + \int_{\partial\Omega} \lambda w^2 d \mathcal{H}^{N-1} \le n\int_\Omega  w + n^{\eta+1}\int_{\partial\Omega} w d \mathcal{H}^{N-1}.
	\end{equation*}
	Now observe that on the left hand side our norm appears, while on the right hand side we may apply Young's inequality with weights $(\varepsilon_1, C_{\varepsilon_1})$, $(\varepsilon_2, C_{\varepsilon_2})$ (where $\varepsilon_1,\varepsilon_2>0$ to be chosen) which leads to
	\begin{equation*}\label{exapprox2}
		\|w\|^2_{\lambda,2} \le n\varepsilon_1\int_\Omega  w^2 + n^{\eta+1}\varepsilon_2 \int_{\partial\Omega} w^2 d \mathcal{H}^{N-1} + 	C_{\varepsilon_1} n|\Omega|+ C_{\varepsilon_2} n^{\eta+1}\mathcal{H}^{N-1}(\partial\Omega).
	\end{equation*}	
	Then applying \eqref{trace} one simply gets
		\begin{equation*}
		\|w\|^2_{\lambda,2} \le n\varepsilon_1C_1 \|w\|^2_{\lambda,2} + n^{\eta+1}\varepsilon_2 C_2\|w\|^2_{\lambda,2} + 	C_{\varepsilon_1} n|\Omega|+ C_{\varepsilon_2} n^{\eta+1}\mathcal{H}^{N-1}(\partial\Omega),
			\end{equation*}	
		where we also used that $\|\cdot\|_{\lambda,2}$ and $\|\cdot\|_{H^1(\Omega)}$ are equivalent norms.
	Then fixing $\varepsilon_2$ satisfying $n^{\eta+1}\varepsilon_2 C_2<\frac{1}{2}$ and $\varepsilon_1$ such that $n\varepsilon_1 C_1<\frac{1}{4}$, one deduces that

	$$\|w\|^2_{\lambda,2}\le C_{n}$$

	where $C_n$ is a positive constant which depends on $n$ but it is independent on $w$. Applying again \eqref{trace}, we obtain that a ball in $L^2(\partial \Omega)$ (let's say of radius $R_n$) is invariant for $T$.

	\medskip

	Moreover, since $C_n$ does not depend on $w$, the compactness of the trace embedding and the above argument show that $\overline{T(A)}$ is compact for any $A$ subset of the ball of radius $R_n$ contained in $L^2(\partial\Omega)$.

	\medskip

	For the continuity we let $v_k\in L^2(\partial\Omega)$ which converges to $v$ in $L^2(\partial\Omega)$ as $k\to\infty$ and we consider $T(v_k)=w_k\big|_{\partial\Omega}$ that is $w_k$ satisfies	
	\begin{equation}
		\label{pbapproxschauder1}
		\begin{cases}
			\displaystyle -\Delta w_k = f_n & \text{ in }\Omega,\\
			\displaystyle \frac{\partial w_k}{\partial \nu} + \lambda w_k = \frac{g_n}{\left(|v_k|+\frac{1}{n}\right)^\eta} & \text{ on } \partial\Omega.
		\end{cases}
	\end{equation}
	Reasoning as for the proof of the invariance, one deduces that $w_k$ is bounded in $H^1(\Omega)$ with respect to $k$. This is sufficient to pass to the limit any term in the weak formulation of \eqref{pbapproxschauder1} using weak convergence in $H^1(\Omega)$ and strong convergence in $L^2(\partial\Omega)$ of $w_k$ to a function $w$
	as $k\to\infty$.
	
	\medskip
	
	As already mentioned above, we are now able to deduce from the Schauder fixed point Theorem that there exists a solution $u_n$ to \eqref{pbapprox}. It follows by taking $u_n^-$ that $u_n \ge 0$ almost everywhere in $\Omega$.  This concludes the proof.
\end{proof}

\begin{remark}
	Let us just underline that, instead of finding the fixed point, one could have proven Lemma \ref{lem_ex} by minimizing the following functional
	\[
	I(u)=\frac12\int_\Omega|\nabla u|^2+\frac12\int_{\partial\Omega}\lambda u^2 d \mathcal{H}^{N-1}-\int_\Omega f_nu-\int_{\partial\Omega}g_n\log\left(\frac1n+u^+\right)d \mathcal{H}^{N-1}, \ u\in H^1(\Omega).
	\]
\end{remark}

Let us now show that the sequence $u_n$ is nondecreasing in $n$.

\begin{lemma}\label{lem_monotonia}
	Under the assumptions of Lemma \ref{lem_ex} let $u_n$ be a solution to \eqref{pbapprox}. Then the sequence $u_n$ is nondecreasing with respect to $n$. Moreover there exists $\overline{c}>0$ such that
	\begin{equation}\label{boundfrombelow}
		u_n(x) \ge \overline{c}>0 \ \text{ for $\mathcal{H}^{N-1}$ almost every } x\in \partial\Omega \text{ and for any $n\in \mathbb{N}$}.
	\end{equation}
\end{lemma}
\begin{proof}
	Let us take $(u_n- u_{n+1})^+$ as a test function in the difference of the weak formulations solved, respectively, by $u_n$ and by $u_{n+1}$. Then one yields to
	\begin{equation*}
	\begin{aligned}
		&\int_{\Omega} |\nabla (u_n- u_{n+1})^+|^2 + \int_{\partial\Omega} \lambda \left((u_n- u_{n+1})^+\right)^2d \mathcal{H}^{N-1}
		\le \int_\Omega \left(f_n - f_{n+1}\right)(u_n- u_{n+1})^+
		\\
		&+ \int_{\partial\Omega} \left(\frac{g_n}{(u_n+\frac{1}{n})^\eta} - \frac{g_{n+1}}{(u_{n+1}+\frac{1}{n+1})^\eta}\right)(u_n- u_{n+1})^+d \mathcal{H}^{N-1},
	\end{aligned}	
\end{equation*}
which implies that
	\begin{equation}\label{mono1}
	\begin{aligned}
		&\int_{\Omega} |\nabla (u_n- u_{n+1})^+|^2 + \int_{\partial\Omega} \lambda \left((u_n- u_{n+1})^+\right)^2d \mathcal{H}^{N-1}
		\le
		\\
		&\int_{\partial\Omega} g_{n+1} \left(\frac{1}{(u_n+\frac{1}{n+1})^\eta} - \frac{1}{(u_{n+1}+\frac{1}{n+1})^\eta}\right)(u_n- u_{n+1})^+d \mathcal{H}^{N-1} \le 0.
	\end{aligned}	
\end{equation}
Equation \eqref{mono1} gives that $\|(u_n-u_{n+1})^+\|_{\lambda, 2}=0$ which means that $u_{n+1} \ge u_n$  $\mathcal{H}^{N-1}$ almost everywhere on $\partial\Omega$ and almost everywhere in $\Omega$.

\medskip

To prove \eqref{boundfrombelow} let us observe that it follows from classical results that there exists $v\in C^1(\overline{\Omega})$ nonnegative solution to
\begin{equation}
	\label{pbsottosol}
	\begin{cases}
		\displaystyle -\Delta v = f_1 & \text{ in }\Omega,\\
		\displaystyle \frac{\partial v}{\partial \nu} + ||\lambda||_{L^\infty(\partial\Omega)} v = 0 & \text{ on } \partial\Omega.
	\end{cases}
\end{equation}
A contradiction argument, using the Hopf Lemma \cite[Theorem $2$]{vaz}, shows that $v>0$ in $\bar\Omega$.
Moreover, analogously to the monotonicity's proof of $u_n$ in $n$, one can show that
\begin{equation*}
	u_n\ge v \ \text{ for $\mathcal{H}^{N-1}$ almost every } x\in \partial\Omega \text{ and for any $n\in \mathbb{N}$}.
\end{equation*}
Since $v$ is continuous and strictly positive on $\partial\Omega$ this shows that
 \begin{equation*}
 	u_n\ge v> \min_{\partial\Omega}v = \overline{c} \text{ for $\mathcal{H}^{N-1}$ almost every } x\in \partial\Omega \text{ and for any $n\in \mathbb{N}$}.
 \end{equation*}
\end{proof}

Let us explicitly underline that in the previous proof the fact that $\lambda$ is bounded and not identically null plays an essential role.

\medskip

Let us show some a priori estimates on $u_n$ with respect to $n$.
\begin{lemma}\label{lem_priorireg}
	Let $0\le f\in L^{\frac{2N}{N+2}}(\Omega)$, $0 \le \lambda \in L^{\infty}(\partial\Omega)$ not identically null and $0\le g\in L^r(\partial\Omega)$ with $r$ satisfying \eqref{erresecreg}. Let $u_n$ be a solution to \eqref{pbapprox} then $u_n$ is bounded in $H^1(\Omega)$ with respect to $n$.
\end{lemma}
\begin{proof}
	Let us take $u_n$ as a test function in \eqref{weakfor}, obtaining
\begin{equation}\label{stimapriori1}
	\int_\Omega |\nabla u_n|^2 + \int_{\partial\Omega} \lambda u_n^2 d \mathcal{H}^{N-1} = \int_{\Omega} f_n u_n  + \int_{\partial\Omega} \frac{g_n u_n}{(u_n+\frac{1}{n})^\eta}d \mathcal{H}^{N-1}.
\end{equation}
	For the first term on the right-hand of \eqref{stimapriori1}, it follows from the H\"older and Sobolev inequalities that
	\begin{equation}\label{stimapriori2}
		\int_{\Omega} f_n u_n \le \|f\|_{L^{\frac{2N}{N+2}}(\Omega)} \|u_n\|_{L^{\frac{2N}{N-2}}(\Omega)} \le \mathcal{S}_2  \|f\|_{L^{\frac{2N}{N+2}}(\Omega)} \| u_n\|_{H^1(\Omega)},
	\end{equation}
	where $\mathcal{S}_2$ is the best constant in the Sobolev inequality for functions in $H^1(\Omega)$.
	For the second term in the right-hand of \eqref{stimapriori1} we observe that, if $\eta \ge 1$, one can simply estimate as
	\begin{equation}\label{stimapriori3}
		\int_{\partial\Omega} \frac{g_n u_n}{(u_n+\frac{1}{n})^\eta} d \mathcal{H}^{N-1}\le \frac{\|g\|_{L^1(\partial\Omega)}}{\overline{c}^{\eta-1}} .
	\end{equation}
	Otherwise if $\eta<1$ it follows from the H\"older inequality and from the choice of $r$ that
	\begin{equation*}
		\int_{\partial\Omega} \frac{g_n u_n}{(u_n+\frac{1}{n})^\eta}  d \mathcal{H}^{N-1}\le \|g\|_{L^r(\partial\Omega)} \|u_n\|^{1-\eta}_{L^{(1-\eta)\frac{r}{r-1}}(\partial\Omega)} =  \|g\|_{L^r(\partial\Omega)} \|u_n\|^{1-\eta}_{L^{\frac{2(N-1)}{N-2}}(\partial\Omega)},
	\end{equation*}
	which, applying \eqref{trace}, gives
	\begin{equation}\label{stimapriori4}
	\int_{\partial\Omega} \frac{g_n u_n}{(u_n+\frac{1}{n})^\eta} d \mathcal{H}^{N-1} \le c\|g\|_{L^r(\partial\Omega)}\| u_n\|_{H^1(\Omega)}^{1-\eta},
	\end{equation}	
	where $c$ does not depend on $n$. Therefore, gathering \eqref{stimapriori2} and \eqref{stimapriori3} in \eqref{stimapriori1}, one gets that for $\eta\ge 1$ it holds
	\begin{equation*}\label{stimafinale1}
	 	\| u_n\|^2_{\lambda,2} \le c\mathcal{S}_2\|f\|_{L^{\frac{2N}{N+2}}(\Omega)} \| u_n\|_{H^1(\Omega)} + \frac{\|g\|_{L^1(\partial\Omega)}}{\overline{c}^{\eta-1}}.
	\end{equation*}
Otherwise, if $\eta<1$, one uses \eqref{stimapriori4} in place of \eqref{stimapriori3} in order to deduce that

	\begin{equation*}\label{stimafinale2}	
		\begin{aligned}
\|u_n\|^2_{\lambda,2} &\le c\mathcal{S}_2\|f\|_{L^{\frac{2N}{N+2}}(\Omega)} \|u_n\|_{H^1(\Omega)}
	\\
	&+ c\|g\|_{L^r(\partial\Omega)} \|u_n\|_{H^1(\Omega)}^{1-\eta}.
		\end{aligned}
	\end{equation*}
Recalling that $\|\cdot\|_{\lambda,2}$ and $\|\cdot\|_{H^1(\Omega)}$ are equivalent norms
and applying the Young inequality one simply deduces that $u_n$ is bounded in $H^1(\Omega)$ with respect to $n$.
\end{proof}

We are ready to prove Theorem \ref{teoexreg}.

\begin{proof}[Proof of Theorem \ref{teoexreg}]
	
Let $u_n$ be a solution to \eqref{pbapprox} whose existence is guaranteed from Lemma \ref{lem_ex}. Then it follows from Lemma \ref{lem_priorireg} that $u_n$ is bounded in $H^1(\Omega)$ with respect to $n$. Moreover, classical embedding results give that $u_n$ (up to not relabeled subsequences) converges to a function $u$ in $L^q(\Omega)$ for any $q<\frac{2N}{N-2}$ and in $L^t(\partial\Omega)$ for any $t<\frac{2(N-1)}{N-2}$ as $n\to\infty$. This is sufficient to pass to the limit the first and the second term of \eqref{weakfor}. The third term simply passes to the limit in $n$. For the fourth term we apply the Lebesgue Theorem since

\begin{equation*}
	\frac{g_n}{(u_n+\frac{1}{n})^\eta} \le  \frac{g}{\bar c^\eta},
\end{equation*}
$\mathcal{H}^{N-1}$ almost everywhere on $\partial\Omega$. This concludes the proof.

\end{proof}

\begin{remark}
	Let us stress once again that in the current section we heavily used that $\lambda \in L^\infty(\partial\Omega)$ and that is not identically null. Indeed, these facts allowed to exploit the maximum principle deducing that the approximating solutions are bounded away from zero on the boundary of $\Omega$. In some sense, if $\lambda$ is bounded, problem \eqref{pbmainreg} seems to be non-singular. On the other hand, in the next section, under more general assumptions this procedure can not be carried over and we need to control the singularity through the use of suitable test functions.
\end{remark}

\section{$L^1$-data and entropy solutions}
\label{sec_entropy}

In this section let $\Omega$ be an open bounded set of $\mathbb{R}^N$ ($N\ge 2$) with Lipschitz boundary.
Here we generalize the results obtained for \eqref{pbmainreg} in the previous section to the following more general problem:
\begin{equation}
	\label{pbmaingeneral}
	\begin{cases}
		\displaystyle -\operatorname{div}(a(x,\nabla u)) = f & \text{ in }\Omega,\\
		u\ge 0 & \text{ in }\Omega,\\
		\displaystyle a(x,\nabla u)\cdot \nu + \lambda \sigma(u) = h(u)g & \text{ on } \partial\Omega,
	\end{cases}
\end{equation}
where $\displaystyle{a(x,\xi):\Omega\times\mathbb{R}^{N} \to \mathbb{R}^{N}}$ is a Carath\'eodory function such that:
\begin{align}
	&a(x,\xi)\cdot\xi\ge \alpha |\xi|^{p} \quad \text{for some } \alpha>0,
	\label{cara1}\\
	&|a(x,\xi)|\le \beta(z(x) + |\xi|^{p-1}) \quad \text{for some } \beta>0 \text{ and } 0\le z\in L^{\frac{p}{p-1}}(\Omega),
	\label{cara2}\\
	&(a(x,\xi) - a(x,\xi^{'} )) \cdot (\xi -\xi^{'}) > 0,
	\label{cara3}	
\end{align}
for $1<p<N$, for almost every $x$ in $\Omega$ and for every $\xi\neq\xi^{'}$ in $\mathbb{R}^N$.
Here $0\le f\in L^{1}(\Omega)$ and $0\le \lambda\in L^1(\partial\Omega)$ is not identically null in $\Omega$. Finally $0\le g\in L^1(\partial\Omega)$.  Here the function $h$ is continuous on $(0,\infty)$ which is finite outside the origin and it can blow up at zero satisfying the following growth condition:
\begin{equation}\label{k}
	\exists\;\eta \ge 0, c_1,s_1>0: \  h(s) \le \frac{c_1}{s^{\eta}} \quad \text{if }s\leq s_1.
\end{equation}
In what follows we denote as $\displaystyle h(0):=\lim_{s\to 0} h(s)$ which exists.
Moreover we require that
\begin{equation}\label{klim}
	%h(\infty):= \lim_{s\to\infty}h(s),
	\limsup_{s\to\infty}h(s)<\infty.
\end{equation}
Finally the function $\sigma$ is continuous and such that:
\begin{equation}\label{sigma}
	\sigma(s) \ge s^{p-1} \text{ if } s\ge 0 \text{ and  } \sigma(0)=0.
\end{equation}
\begin{remark}\label{rem_dati}
	Under the above assumptions we are not in position to reason as in Section \ref{sec_reg}; in particular we can not deduce the existence of subsolution as \eqref{pbsottosol} for the approximating sequence which is bounded from below by a positive constant at the boundary of $\Omega$.

	Moreover, besides the unboundedness of $\lambda$, we also require $\sigma$ and $h$ to be functions where no monotonicity is assumed. Finally $f$ and $g$ are merely integrable functions. All the arguments above force us to employ a different technique to pass to the limit in the approximation sequence.
\end{remark}

As we will see, since we also deal with uniqueness of solutions under some restrictive hypotheses,  the entropy setting better adapts with $L^1$-data. Firstly we precisely set what we mean by entropy solution for problem \eqref{pbmaingeneral} and, after that, we make some comments on the notion of solution.

\begin{defin}\label{entropydef}
	A measurable function $u$ which is almost everywhere finite in $\Omega$ and such that $T_k(u)\in W^{1,p}(\Omega)$ for all $k>0$ is an entropy solution to \eqref{pbmaingeneral} if $a(x, \nabla T_k(u))\in L^{\frac{p}{p-1}}(\Omega)^N$, $\lambda \sigma(u),h(u)g \in L^1(\partial\Omega)$
	and it holds
	\begin{equation}\label{def_ent}
		\begin{aligned}
			&\int_\Omega a(x,\nabla u)\cdot \nabla T_k(u-v) +  \int_{\partial\Omega} \lambda \sigma(u)T_k(u-v)d \mathcal{H}^{N-1} \\
		&= \int_\Omega fT_k(u-v) + \int_{\partial\Omega} h(u)gT_k(u-v)d \mathcal{H}^{N-1}	
		\end{aligned}
	\end{equation}
	for all $v \in W^{1,p}(\Omega)\cap L^\infty(\Omega)$ and all $k>0$.
\end{defin}

\begin{remark}
	Let us clarify the meaning of $\nabla u$ since we do not necessarily deal with functions in $W^{1,1}(\Omega)$.
	
	It is classical nowadays that from Lemma $2.1$ of \cite{b6} there exists  a unique measurable function $v$ such that
	$$\nabla T_k(u) = v \chi_{\{|u_n|\le k\}}$$
	for almost every $x\in \Omega$ and for every $k>0$. Moreover it is shown that $u\in W^{1,1}(\Omega)$ if and only if $v\in L^1(\Omega)$ and $v=\nabla u$ in the usual distributional sense.
	
	This motivates the choice of referring to the above cited function $v$ when dealing to the gradient of a function $u$ having only its truncations in a Sobolev space.
\end{remark}

\begin{remark}
	Let us stress that the first term on the left-hand of \eqref{def_ent} is finite. Indeed, $\nabla T_k(u-v)$ is different from zero only on $\{|u-v|<k\}$ where $|u|< \|v \|_{L^\infty(\Omega)} + k=: M$. Hence, since $T_M(u)\in  W^{1,p}(\Omega)$, we deduce $a(x,\nabla T_M(u))\in L^{\frac{p}{p-1}}(\Omega)^N$ and $\nabla T_k(T_M(u)-v)\in L^p(\Omega)^N$. Clearly, it is simple to convince that also all the other terms are well defined.
	Moreover, let us explicitly underline that it is easy to see that a solution $u\in W^{1,p}(\Omega)$ satisfying a formulation analogous to \eqref{weakdefregformulation} is also an entropy solution.
	Conversely, any entropy solution $u$ belonging to $W^{1,p}(\Omega)$ is also a solution in the sense of \eqref{weakdefregformulation} if $f\in L^{\frac{Np}{Np-N+p}}(\Omega)$.
\end{remark}

\medskip

Hence we state the existence result to \eqref{pbmaingeneral}.

\begin{theorem}\label{teoexgeneral}
	Let $a$ satisfy \eqref{cara1}, \eqref{cara2} and \eqref{cara3}. Let $0\le f\in L^{1}(\Omega)$, let $0\le \lambda \in L^{1}(\partial\Omega)$ not identically null and let  $0\le g\in L^1(\partial\Omega)$. Finally let $h$ satisfy \eqref{k} and \eqref{klim}. Then there exists a nonnegative entropy solution $u$ to \eqref{pbmaingeneral} such that $u\in M^{\frac{N(p-1)}{N-p}}(\Omega)$, $u\in M^{\frac{(N-1)(p-1)}{N-p}}(\partial\Omega)$ and $|\nabla u|\in M^{\frac{N(p-1)}{N-1}}(\Omega)$.
\end{theorem}

Under some restrictive assumptions we show that there is at most one entropy solution to \eqref{pbmaingeneral}.
\begin{theorem}\label{teouniqueent}
Let $a$ satisfy \eqref{cara1}, \eqref{cara2} and \eqref{cara3} and let $\lambda, g \ge 0$ $\mathcal{H}^{N-1}$ almost everywhere on $\partial\Omega$. Finally assume that $\sigma(s)$ is increasing and $h(s)$ is nonincreasing with respect to $s$. Then there is at most one entropy solution to \eqref{pbmaingeneral}.
\end{theorem}

\subsection{Approximation scheme and a priori estimates}

Once again we work by approximation through the following scheme:
\begin{equation}
	\label{pbapproxgen}
	\begin{cases}
		\displaystyle -\operatorname{div}( a(x,\nabla u_n)) = f_n & \text{ in }\Omega,\\
		\displaystyle a(x,\nabla u_n)\cdot \nu  + \lambda_n \sigma_n(u_n) = h_n(u_n)g_n & \text{ on } \partial\Omega,
	\end{cases}
\end{equation}
where $f_n:= T_n(f), \lambda_n:=T_n(\lambda), \sigma_n(s) := T_n(\sigma(s)), h_n(s) := T_n(h(s)), g_n:=T_n(g)$. We start proving the existence of a solution $u_n$ to \eqref{pbapproxgen}.
\begin{lemma}\label{lem_exgen}
	Let $a$ satisfy \eqref{cara1}, \eqref{cara2} and \eqref{cara3}. Let $0\le f\in L^{1}(\Omega)$, $0\le \lambda \in L^{1}(\partial\Omega)$ not identically null and let $0\le g\in L^1(\partial\Omega)$. Let $h$ satisfy \eqref{k} and \eqref{klim} and finally let $\sigma$ satisfy \eqref{sigma}. Then there exists a nonnegative weak solution $u_n\in H^1(\Omega)\cap L^\infty(\Omega)$ to \eqref{pbapproxgen}.
\end{lemma}
\begin{proof}
Let us provide a very brief idea of the proof.\\
The existence of $w\in W^{1,p}(\Omega)$ nonnegative solution to
	\begin{equation*}
		\label{fixedpoint}
		\begin{cases}
			\displaystyle -\operatorname{div}( a(x,\nabla w)) = f_n & \text{ in }\Omega,\\
			\displaystyle a(x,\nabla w)\cdot \nu  + \lambda_n \sigma_n(|w|) = h_n(|v|)g_n & \text{ on } \partial\Omega,
		\end{cases}
	\end{equation*}
where $v\in L^p(\partial\Omega)$ follows from \cite{ll}. Then a very similar reasoning to the one of Lemma \ref{lem_ex} gives that the application $T: L^p(\partial\Omega)\mapsto L^p(\partial\Omega)$ such that $T(v)=w\big|_{\partial\Omega}$ has a fixed point. The main difference, apart from the estimates in which one heavily uses \eqref{cara1}, lies in the continuity request; if  $w_k\big|_{\partial\Omega}= T(v_k)$ and $v_k$ converges to $v$ in $L^p(\partial\Omega)$, in this case one has also to show that  $\nabla w_k$ converges almost everywhere to some $\nabla w$ in $\Omega$ to pass to the limit the principal operator in order to have that $w\big|_{\partial\Omega}= T(v)$. Since $f_n$ is independent of $k$ and $h_n(|v|)g_n\le n^2$ one can reason as in Lemma \ref{lem_conv} below in order to deduce the desired convergence. Then the continuity part is analogous to the one proven in Lemma \ref{lem_ex}.
\end{proof}

Let us now show some a priori estimates for $u_n$ in $n$.

\begin{lemma}\label{lem_priorigenlambda}
	Under the assumptions of Lemma \ref{lem_exgen} let $u_n$ be a solution to \eqref{pbapproxgen}. Then it holds:
	\begin{equation}\label{stimaassorbimentogen}
		\int_{\{u_n>t\}} \lambda_n \sigma_n(u_n) d \mathcal{H}^{N-1} \le \int_{\{u_n>t\}} f_n + \int_{\{u_n>t\}} h_n(u_n)g_n d \mathcal{H}^{N-1}, \ \ \forall t>0.
	\end{equation}
It holds that $\lambda_n \sigma_n(u_n)$ is bounded in $L^1(\partial\Omega)$, $u_n$ is bounded in $M^{\frac{N(p-1)}{N-p}}(\Omega)$ and $M^{\frac{(N-1)(p-1)}{N-p}}(\partial\Omega)$ and $|\nabla u_n|$ is bounded in $M^{\frac{N(p-1)}{N-1}}(\Omega)$ with respect to $n$. Moreover $h_n(u_n)g_n$ is bounded in $L^1(\partial\Omega)$ with respect to $n$. In particular it holds
\begin{equation}\label{stimaTk}
	\|T_k(u_n)\|_{W^{1,p}(\Omega)} \le Ck^{\frac{1}{p}},
\end{equation}
for some positive constant $C$ which does not depend on $n$.
\end{lemma}
\begin{proof}
	Let $t,\varepsilon>0$ and let us take $\phi_{t,\varepsilon}(u_n)$ ($\phi_{t,\varepsilon}$ is defined in \eqref{phit}) as a test function in the weak formulation of  \eqref{pbapproxgen} yielding to (recall \eqref{cara1})

	\begin{equation}\label{lemmapriori1gen}
	\begin{aligned}
		&\alpha\int_\Omega |\nabla u_n|^p \phi'_{t,\varepsilon}(u_n)  + \int_{\partial\Omega} \lambda_n \sigma_n(u_n) \phi_{t,\varepsilon}(u_n) d \mathcal{H}^{N-1}
		\\
		&\le \int_{\Omega} f_n \phi_{t,\varepsilon}(u_n) + \int_{\partial\Omega} h_n(u_n)g_n\phi_{t,\varepsilon}(u_n) d \mathcal{H}^{N-1}.
	\end{aligned}
	\end{equation}

	Since $\phi_{t,\varepsilon}$ is nondecreasing, it is different from zero only on $\{u_n>t\}$ one gets
	$$
	\int_{\partial\Omega} \lambda_n \sigma_n(u_n) \phi_{t,\varepsilon}(u_n) d \mathcal{H}^{N-1} \le \int_{\{u_n>t\}} f_n + \int_{\{u_n>t\}} h_n(u_n)g_n d \mathcal{H}^{N-1},
	$$
	and \eqref{stimaassorbimentogen} is obtained by an application of the Fatou Lemma as $\varepsilon\to 0^+$. Let also highlight that the right hand of \eqref{lemmapriori1gen} is bounded by a constant which is independent of $n$:
	$$\int_{\{u_n>t\}} f_n + \int_{\{u_n>t\}} h_n(u_n)g_n d \mathcal{H}^{N-1} \le
	\int_{\Omega} f+ \sup_{s\in(t,\infty)}h(s) \int_{\partial\Omega} g d \mathcal{H}^{N-1}=c(t).$$
	This implies
	\begin{equation*}
	\begin{aligned}\int_{\partial\Omega} \lambda_n \sigma_n(u_n) d \mathcal{H}^{N-1}&=\int_{\{u_n\le 1\}} \lambda_n \sigma_n(u_n)  d \mathcal{H}^{N-1}+\int_{\{u_n>1\}} \lambda_n \sigma_n(u_n)  d \mathcal{H}^{N-1}
	\\
	&\le \max_{s\in[0,1]}\sigma(s)\int_{\partial\Omega} \lambda d \mathcal{H}^{N-1} +c(1)
	\end{aligned}
	\end{equation*}
	so that $\lambda_n \sigma(u_n)$ is bounded in $L^1(\partial\Omega)$ with respect to $n$.
	\\
	Now we focus on the Sobolev estimate for $u_n$. Let us take $T_k(u_n)-k$ as a test function in the weak formulation of  \eqref{pbapproxgen} yielding to
	\begin{equation*}
		\alpha\int_\Omega |\nabla T_k(u_n)|^p + \int_{\partial\Omega} \lambda_n \sigma_n(u_n) (T_k(u_n)-k) d \mathcal{H}^{N-1} \le 0,
	\end{equation*}
	which means that
	\begin{equation*}
		\int_\Omega |\nabla T_k(u_n)|^p \le
		\frac{k}{\alpha}\int_{\partial\Omega} \lambda_n \sigma_n(u_n)d \mathcal{H}^{N-1}  \le Ck,
	\end{equation*}
	where $C$ does not depend on $n$ since $\lambda_n \sigma_n(u_n)$ is bounded in $L^1(\partial\Omega)$ with respect to $n$.
	Then we have shown that
	\begin{equation*}\label{stimatroncate}
		\int_\Omega |\nabla T_k(u_n)|^p + \int_{\partial\Omega} \lambda_n \sigma_n(u_n)T_k(u_n) \le Ck, \ \ \forall k>0.
	\end{equation*}
	Thus, recalling \eqref{sigma} and that $\lambda_n\ge \lambda_1$, the previous implies that for any $k>0$ \eqref{stimaTk} holds
	since $\|\cdot\|_{\lambda_1,p}$ and $\|\cdot\|_{W^{1,p}(\Omega)}$ are equivalent.
	
	It follows from classical arguments that $u_n$ is bounded in $M^{\frac{N(p-1)}{N-p}}(\Omega)$ and $|\nabla u_n|$ is bounded in $M^{\frac{N(p-1)}{N-1}}(\Omega)$ with respect to $n$ (see for instance \cite{b6}).
	
	\medskip
	
	Here we briefly sketch the boundedness in $M^{\frac{(N-1)(p-1)}{N-p}}(\partial\Omega)$. It follows from \eqref{stimaTk} and \eqref{trace} that
	
	\begin{equation*}
			 k\mathcal{H}^{N-1}(\{x\in \partial\Omega: u_n\ge k\})^{\frac{N-p}{(N-1)p}} \le  \|T_k(u_n)\|_{L^{\frac{(N-1)p}{N-p}}(\partial\Omega)}  \le Ck ^{\frac{1}{p}},
	\end{equation*}
	which implies that
		\begin{equation*}
		\mathcal{H}^{N-1}(\{x\in \partial\Omega: u_n\ge k\}) \le  \frac{C}{k^{\frac{(N-1)(p-1)}{N-p}}},
	\end{equation*}
	namely $u_n$ is bounded in $M^{\frac{(N-1)(p-1)}{N-p}}(\partial\Omega)$ with respect to $n$.
	
	\medskip
	
	Now we show that $h_n(u_n)g_n$ is bounded in $L^1(\partial\Omega)$ with respect to $n$.	
	Let us take $V_\delta(u_n)$ ($V_\delta$ is defined in \eqref{Vdelta}) as a test function in the weak formulation of \eqref{pbapproxgen}. This takes to
	\begin{equation*}
		\begin{aligned}
			&\int_\Omega a(x,\nabla u_n)\cdot \nabla u_n V'_\delta(u_n)  + \int_{\partial\Omega} \lambda_n \sigma_n(u_n) V_\delta(u_n) d \mathcal{H}^{N-1}
			\\
			&= \int_{\Omega} f_n V_\delta(u_n) + \int_{\partial\Omega} h_n(u_n)g_n V_\delta(u_n) d \mathcal{H}^{N-1}.
		\end{aligned}
	\end{equation*}
	Now recalling that $f_n\ge 0$, $V'_\delta(s)\le 0$ and that \eqref{cara1} holds, the previous implies
	\begin{equation*}
		\begin{aligned}
			\int_{\{u_n\le \delta\}} h_n(u_n)g_n d \mathcal{H}^{N-1} &\le \int_{\partial\Omega} h_n(u_n)g_n V_\delta(u_n) d \mathcal{H}^{N-1}
			\\
			&\le \int_{\partial\Omega} \lambda_n \sigma_n(u_n) V_\delta(u_n) d \mathcal{H}^{N-1} \le C,
		\end{aligned}
	\end{equation*}
	since $\lambda_n \sigma_n(u_n)$ is bounded in $L^1(\partial\Omega)$. This concludes the proof.
\end{proof}

\subsection{Convergence results}
This subsection is devoted to the proof of the convergence results concerning $u_n$ which are needed to prove Theorem \ref{teoexgeneral}.

\begin{lemma}\label{lem_conv}
	Under the assumptions of Lemma \ref{lem_exgen} let $u_n$ be a solution to \eqref{pbapproxgen}. Then $u_n$ converges (up to a subsequence) almost everywhere in $\Omega$ and $\mathcal{H}^{N-1}$ almost everywhere in $\partial\Omega$ as $n\to\infty$ to a function $u$ which is almost everywhere finite in $\Omega$ and on $\partial\Omega$.
	\\
	Moreover $\lambda_n\sigma_n(u_n)$ and $h_n(u_n)g_n$ converge in $L^1(\partial\Omega)$ respectively to $\lambda\sigma(u)$ and $h(u)g$ as $n\to\infty$.
	\\
	Finally $T_k(u_n)$ converges to $T_k(u)$ strongly in $W^{1,p}(\Omega)$ as $n\to\infty$ and for every $k>0$.
\end{lemma}
\begin{proof}
It follows from Lemma \ref{lem_priorigenlambda} that $u_n$ is bounded in $M^{\frac{N(p-1)}{N-p}}(\Omega)$ with respect to $n$ and $T_{k}(u_{n})$ is bounded in $W^{1,p}(\Omega)$ for any $k>0$, then $u_{n}$ converges (up to not relabeled subsequences) almost everywhere to a function $u$ such that $T_{k}(u)\in W^{1,p}(\Omega)$ and $u$ is almost everywhere finite in $\Omega$. Moreover, by a suitable compactness argument (in $n$) for $T_{k}(u_{n})$ on $\de\Omega$, $u_n$ converges (up to a subsequence) $\mathcal{H}^{N-1}$ almost everywhere to $u$ in $\partial\Omega$. The function $u$ is $\mathcal H^{N-1}-$a.e. finite on $\de\Omega$ since $u_n$ is bounded in $M^{\frac{(N-1)(p-1)}{N-p}}(\partial\Omega)$ as proven in Lemma \ref{lem_priorigenlambda}.

\medskip

Now observe that   \eqref{stimaassorbimentogen} implies that $\lambda_n\sigma_n(u_n)$ is equiintegrable and it converges to $\lambda\sigma(u)$ in $L^1(\partial\Omega)$ as $n\to\infty$.

\medskip

Now let us show that $h_n(u_n)g_n$ converges in $L^1(\partial\Omega)$ to $h(u)g$ as $n\to\infty$. If $h(0)<\infty$ then this is obvious; hence, without loss of generality, we assume that $h(0) =\infty$. Firstly observe that $h(u)g \in L^1(\partial \Omega)$; indeed it follows from the weak formulation of \eqref{pbapproxgen} that
$$\int_{\partial\Omega} h_n(u_n)g_n d \mathcal{H}^{N-1} \le \int_{\partial\Omega} \lambda_n \sigma_n(u_n) d \mathcal{H}^{N-1} \le C,
$$
thanks to Lemma \ref{lem_priorigenlambda}.
Then an application of the Fatou Lemma  gives $h(u)g \in L^1(\partial\Omega)$ which also means
\begin{equation}\label{uzero}
	\{u=0\}\subset \{g=0\}
	\text{ if } h(0)=\infty,
\end{equation}
up to a set of zero $\mathcal{H}^{N-1}$ measure set.

Now let us take $V_\delta(u_n)$ as a test function in the weak formulation of \eqref{pbapproxgen} yielding to
\begin{equation*}\label{convk}
		\int_{\{u_n\le \delta\}} h_n(u_n)g_n d \mathcal{H}^{N-1} \le \int_{\partial \Omega} h_n(u_n)g_n V_\delta(u_n) d \mathcal{H}^{N-1} \le  \int_{\partial\Omega} \lambda_n \sigma_n(u_n) V_\delta(u_n) \ d \mathcal{H}^{N-1},
\end{equation*}
where we dropped a non-positive term. Now one can simply take $n\to\infty$ and $\delta\to 0^+$, obtaining that
\begin{equation}\label{limkzero}
	\begin{aligned}
		\lim_{\delta\to 0^+}\limsup_{n\to\infty}\int_{\{u_n\le \delta\}} h_n(u_n)g_n d \mathcal{H}^{N-1}
		\le  \int_{{\{u=0\}}} \lambda \sigma(u) V_\delta(u) \ d \mathcal{H}^{N-1}=0,
	\end{aligned}
\end{equation}
since $\sigma(0)=0$.

Now consider $\delta\not\in \{t: |\{u=t\}|>0\}$, which is admissible since it is a countable set, and split the singular term as
\begin{equation}
	\begin{aligned}
		\label{split}
		\int_{\partial\Omega}h_n(u_n)g_n d \mathcal{H}^{N-1} &= \int_{\{u_n\leq \delta\}} h_n(u_n)g_n d \mathcal{H}^{N-1}\\
		&+ \int_{\{u_n > \delta\}} h_n(u_n)g_n d \mathcal{H}^{N-1}.
	\end{aligned}
\end{equation}
For the first term of \eqref{split} it holds \eqref{limkzero} as $n\to\infty$ and $\delta \to 0^+$.

For the second term in the right-hand of the previous one can apply the Lebesgue Theorem since
$$
h_n(u_n)g_n\chi_{\{u_n\ge \delta\}}\leq \sup_{s\in (\delta,\infty)}h(s)g \in L^1(\partial\Omega),
$$
yielding to
\begin{equation*}
	\lim_{n\to\infty} \int_{\{u_n > \delta\}} h_n(u_n)g_n d \mathcal{H}^{N-1}= \int_{\{u > \delta\}} h(u)g d \mathcal{H}^{N-1}.
\end{equation*}

Then, since $h(u)g\in L^1(\partial\Omega)$, one can apply once again the Lebesgue Theorem in order to get
\begin{equation*}
	\label{lim6}
	\lim_{\delta\to 0^+}\lim_{n\to\infty} \int_{\{u_n > \delta\}} h_n(u_n)g_n d \mathcal{H}^{N-1} = \int_{\partial\Omega} h(u)g d \mathcal{H}^{N-1},
\end{equation*}
thanks to \eqref{uzero}. Since $h_n(u_n)g_n$ is nonnegative, this is sufficient to deduce that it converges to $h(u)g$ in $L^1(\partial\Omega)$ as $n\to\infty$.

\medskip
	
Now we prove that $T_k(u_n)$ converges to $T_k(u)$ strongly in $W^{1,p}(\Omega)$ as $n\to\infty$ and for every $k>0$. Let us take $(T_k(u_n)-T_k(u))V_l(u_n)$ ($l>k$) as a test function in the weak formulation of \eqref{pbapproxgen} yielding to
\begin{equation*}
	\begin{aligned}
		&\int_\Omega (a(x,\nabla T_k(u_n)) - a(x,\nabla T_k(u))) \cdot \nabla (T_k(u_n)-T_k(u))  \\
		&= - \int_{\{k<u_n<2l\}} a(x,\nabla u_n)  \cdot \nabla (T_k(u_n)-T_k(u))  V_l(u_n) \\
		&+ \frac{1}{l}\int_{\{l<u_n<2l\}} a(x,\nabla u_n)\cdot \nabla u_n (T_k(u_n)-T_k(u))  \\
		& + \int_\Omega f_n (T_k(u_n)-T_k(u)) V_l(u_n) + \int_{\partial\Omega} h_n(u_n)g_n (T_k(u_n)-T_k(u)) V_l(u_n)  d \mathcal{H}^{N-1}\\
		& -\int_{\partial\Omega} \lambda_n\sigma_n(u_n) (T_k(u_n)-T_k(u)) V_l(u_n)  d \mathcal{H}^{N-1}
		-\int_\Omega a(x,\nabla T_k(u)) \cdot \nabla (T_k(u_n)-T_k(u))
		\\
		&=: (A)+(B)+(C)+(D)+(E) +(F).	
	\end{aligned}
\end{equation*}
For $(A)$ one has
$$(A) \le \int_\Omega|a(x,\nabla u_{n})|V_l(u_n)|\nabla T_{k}(u)|\chi_{\{u_n>k\}}.$$
Now let underline that $|a(x,\nabla u_{n})|V_l(u_n)$ is bounded in $L^{\frac{p}{p-1}}(\Omega)$ with respect to $n$ and that $\displaystyle |\nabla T_k(u)|\chi_{\{u_n>k\}}$ converges to zero in $L^p(\Omega)$ as $n\to\infty$ then one has
$$
\limsup_{n\to\infty} (A) \le 0.
$$
In order to estimate $(B)$ we take $\phi_{l,l}(u_n)$ as a test function in the weak formulation of \eqref{pbapproxgen}, yielding to
\begin{equation*}
	\frac{1}{l}\int_\Omega a(x,\nabla u_n)\cdot \nabla u_n \le \int_{\Omega}f_n\phi_{l,l}(u_n) + \int_{\partial\Omega} h_n(u_n)g_n \phi_{l,l}(u_n) d \mathcal{H}^{N-1},
\end{equation*}
which simply goes to zero as $n\to \infty$ and $l\to \infty$ since both $f_n$ and $h_n(u_n)g_n$ converges in $L^1(\Omega)$ and in $L^1(\partial\Omega)$ respectively as $n\to\infty$.
Hence one gets that
$$\lim_{l\to\infty}\limsup_{n\to\infty} (B) = 0.$$

Moreover one simply has that

$$\lim_{n\to\infty} (C) = \lim_{n\to\infty} (D) = \lim_{n\to\infty} (E) = 0$$
since $f_n$ converges in $L^1(\Omega)$, and both $h_n(u_n)g_n , \lambda_n\sigma_n(u_n)$ converge in $L^1(\partial\Omega)$ with respect to $n$.
Finally it follows from the weak convergence of $T_k(u_n)$ to $T_k(u)$ as $n\to\infty$ in $W^{1,p}(\Omega)$ that
$$\lim_{n\to\infty} \ (F) = 0.$$

\medskip

Therefore we have proven that
$$\limsup_{n\to\infty} \int_\Omega (a(x,\nabla T_k(u_n)) - a(x,\nabla T_k(u))) \cdot \nabla (T_k(u_n)-T_k(u)) =0,$$
which allows to reason as in the proof of Lemma $5$ of \cite{bmp} in order to conclude the proof.
\end{proof}

\subsection{Existence of an entropy solution}

This section is devoted to the passage to the limit in weak formulation of the approximation scheme \eqref{pbapproxgen}.

\begin{proof}[Proof of Theorem \ref{teoexgeneral}]
Let $u_n$ be a solution to \eqref{pbapproxgen}. Then it follows from Lemma \ref{lem_conv} that $u_n$ converges (up to a subsequence) almost everywhere in $\Omega$ and $\mathcal{H}^{N-1}$ almost everywhere on $\partial\Omega$ to $u$ as $n\to\infty$. Moreover $u$ is almost everywhere finite and $T_k(u) \in W^{1,p}(\Omega)$.

\medskip

Let us firstly observe that $a(x,\nabla T_k(u)) \in L^{\frac{p}{p-1}}(\Omega)^N$ since $T_k(u) \in W^{1,p}(\Omega)$ and thanks to \eqref{cara2}. Let us also note that it follows from Lemma \ref{lem_conv} that $\lambda\sigma(u), h(u)g \in L^1(\partial \Omega)$.

\medskip

Let us prove \eqref{def_ent}. We take $T_k(u_n-v)$ as a test function in the weak formulation of \eqref{pbapproxgen} where $v \in W^{1,p}(\Omega)\cap L^\infty(\Omega)$. Then one obtains
\begin{equation}
\begin{aligned}
	\label{lim1}
	&\int_{\Omega} a(x,\nabla u_n)\cdot\nabla T_k(u_n-v) + \int_{\partial\Omega} \lambda_n\sigma_n(u_n)T_k(u_n-v) d \mathcal{H}^{N-1}
	\\
	&= \int_{\Omega}f_nT_k(u_n-v) + \int_{\partial\Omega} h_n(u_n)g_nT_k(u_n-v) d \mathcal{H}^{N-1},
\end{aligned}
\end{equation}
and we want to pass to the limit \eqref{lim1} as $n\to\infty$ .
For the first term on the left hand side one can write
\begin{equation*}
	\begin{aligned}
		\int_{\Omega} a(x,\nabla u_n)\cdot\nabla T_k(u_n-v) &= \int_{\{|u_n-v|\le k\}} a(x,\nabla u_n)\cdot\nabla u_n
		\\
		&-  \int_{\{|u_n-v|\le k\}} a(x,\nabla u_n)\cdot\nabla v.
	\end{aligned}
\end{equation*}
Let firstly observe that in the previous integrals one has that $u_n\le ||v||_{L^\infty(\Omega)} + k=:M$. Then, since it follows from Lemma \ref{lem_conv} that $T_k(u_n)$ converges strongly to $T_k(u)$ in $W^{1,p}(\Omega)$ as $n\to\infty$ for any $k>0$, one has that  $a(x,\nabla T_M(u_n))$ converges strongly to $a(x,\nabla T_M(u))$ in $L^{\frac{p}{p-1}}(\Omega)^N$ as $n\to\infty$. This is sufficient to deduce that
\begin{equation*}\label{limagg}
	\begin{aligned}
		\lim_{n\to\infty} \int_{\Omega} a(x,\nabla u_n)\cdot\nabla T_k(u_n-v) &=  \int_{\{|u-v|\le k\}} a(x,\nabla u)\cdot\nabla u
		\\
		&-  \int_{\{|u-v|\le k\}} a(x,\nabla u)\cdot\nabla v
		\\
		&=  \int_{\Omega} a(x,\nabla u)\cdot\nabla T_k(u-v).
	\end{aligned}
\end{equation*}

\medskip

Moreover Lemma \ref{lem_conv} also gives that $\lambda_n\sigma_n(u_n), h_n(u_n)g_n$ converge in $L^1(\partial\Omega)$ to $\lambda\sigma(u)$ and $h(u)g$ as $n\to\infty$. This is sufficient to take $n\to\infty$ in the second and in the fourth term of \eqref{lim1}.
The first term on the right-hand simply passes to the limit as $n\to\infty$. This concludes the proof.

\end{proof}

\subsection{Proof of the uniqueness result}
In this section we prove Theorem \ref{teouniqueent}.
\begin{proof}[Proof 	of Theorem \ref{teouniqueent}]
Let $u_1$ and $u_2$ be entropy solutions to problem \eqref{pbmaingeneral} and let us take $v=T_m(u_2)$ in the entropy formulation corresponding to $u_1$ and $v=T_m(u_1)$ in that of $u_2$. Adding up both identities, it leads to
\begin{equation}\label{punto}
	\begin{aligned}
	&\int_{\{|u_1-T_m(u_2)|<k\}}\!\!a(x,\nabla u_1)\cdot \nabla (u_1-T_m(u_2))
	\\
	&+
	\int_{\{|u_2-T_m(u_1)|<k\}}\!\!a(x,\nabla u_2)\cdot \nabla (u_2-T_m(u_1))
	\\
	&+\int_{\partial\Omega}\lambda \sigma(u_1)T_k(u_1-T_m(u_2)) d \mathcal{H}^{N-1}+
	\int_{\partial\Omega}\lambda \sigma(u_2)T_k(u_2-T_m(u_1)) d \mathcal{H}^{N-1}
	\\
	&=\int_\Omega f\big( T_k(u_1-T_m(u_2))+T_k(u_2-T_m(u_1))\big)
	\\
	&+\int_{\partial\Omega}h(u_1) g T_k(u_1-T_m(u_2)) d \mathcal{H}^{N-1} +\int_{\partial\Omega}h(u_2) g T_k(u_2-T_m(u_1)) d \mathcal{H}^{N-1}
	\end{aligned}	
\end{equation}
We let $m\to\infty$ in \eqref{punto}.

For the first two terms of \eqref{punto} one can reason as in Theorem $5.1$ of \cite{b6}, deducing that its liminf as $m\to \infty$ is bigger than
\[\int_{\{|u_1-u_2|<k\}} \big(a(x,\nabla u_1)-a(x,\nabla u_2)\big)\cdot \nabla (u_1-u_2).\]
%More precisely, it was proven that after neglecting certain nonnegative terms, its limit is the above function.

It is easy to handle the other terms thanks to Lebesgue's Theorem. Indeed,
\begin{equation}
	\label{unique}
	\begin{aligned}
	\lim_{m\to\infty}&\int_{\partial\Omega}\lambda \sigma(u_1)T_k(u_1-T_m(u_2)) d \mathcal{H}^{N-1} +
	\int_{\partial\Omega}\lambda \sigma(u_2)T_k(u_2-T_m(u_1)) d \mathcal{H}^{N-1}
	\\
	&=\int_{\partial\Omega}\lambda\big( \sigma(u_1)-\sigma(u_2)\big)T_k(u_1-u_2) d \mathcal{H}^{N-1} \ge0
	\end{aligned}	
\end{equation}
since $\sigma$ is an increasing function. Moreover,
\[
\lim_{m\to\infty}\int_\Omega f\big( T_k(u_1-T_m(u_2))+T_k(u_2-T_m(u_1))\big)
=0
\]
and
\begin{equation*}
	\begin{aligned}
	\lim_{m\to\infty}&\int_{\partial\Omega}h(u_1) g T_k(u_1-T_m(u_2)) d \mathcal{H}^{N-1}+\int_{\partial\Omega}h(u_2) g T_k(u_2-T_m(u_1)) d \mathcal{H}^{N-1}
	\\
	&=\int_{\partial\Omega} g \big(k(u_1)-k(u_2)\big)T_k(u_1-u_2) d \mathcal{H}^{N-1}\le 0
	\end{aligned}	
\end{equation*}
since $k$ is nonincreasing. Therefore, identity \eqref{punto} becomes
\[\int_{\{|u_1-u_2|<k\}} \big(a(x,\nabla u_1)-a(x,\nabla u_2)\big)\cdot \nabla (u_1-u_2) + \int_{\partial\Omega}\lambda\big( \sigma(u_1)-\sigma(u_2)\big)T_k(u_1-u_2) d \mathcal{H}^{N-1} \le0.\]
Now the proof concludes by observing that it follows from \eqref{cara3} and \eqref{unique} that both terms of the previous are zero. This means that $\nabla u_1 = \nabla u_2$ almost everywhere in $\Omega$. Moreover the previous implies that $\lambda u_1 = \lambda u_2$ $\mathcal{H}^{N-1}$ almost everywhere on $\partial\Omega$ since $\sigma(s)$ is increasing in $s$.

Then one has that, for all $k>0$,  $\|T_k(u_1)-T_k(u_2)\|_{\lambda,p} = 0$ and we conclude that
 $u_1=u_2$ almost everywhere in $\Omega$ and $\mathcal{H}^{N-1}$ almost everywhere on $\partial\Omega$.
\end{proof}

\section*{Funding}

The first author is partially supported by the  MIUR-PRIN 2017 grant ``Qualitative and quantitative aspects of nonlinear PDE's'', by GNAMPA of INdAM, by  the FRA Project (Compagnia di San Paolo and Universit\`a degli studi di Napoli Federico II) \verb|000022--ALTRI_CDA_75_2021_FRA_PASSARELLI|.

The second author is partially supported by GNAMPA of INdAM.

The third author is partially supported by CIUCSD (Generalitat Valenciana) under project AICO/2021/223 and by Red de Ecuaciones en Derivadas Parciales no Locales y Aplicaciones of MCI (Spanish) under project RED2022-134784-T.


\begin{thebibliography}{99}

\bibitem{ant} A. Alvino, C. Nitsch and C. Trombetti, A Talenti comparison result for solutions to elliptic problems with Robin boundary conditions. Comm. Pure Appl. Math. 76 (3), 585–603 (2023).


\bibitem{aimt} F. Andreu, N. Igbida, J.M. Maz\'on and J. Toledo,
$L^1$ existence and uniqueness results for quasi-linear elliptic equations with nonlinear boundary conditions. Ann. Inst. Henri Poincaré, Anal. Non Linéaire 24 (1), 61-89 (2007).	

\bibitem{ams} F. Andreu, J.M. Maz\'on, S. Segura de Le\'on and J. Toledo, Quasi-linear elliptic and parabolic equations in $L^1$ with nonlinear boundary conditions, Adv. Math. Sci. Appl. 7 (1), 183–213 (1997). 	
	
\bibitem{b6} P. B\'enilan, L. Boccardo, T. Gallou\"et, R. Gariepy, M. Pierre and J.L. Vázquez, An $L^1$-theory of existence and uniqueness of solutions of nonlinear elliptic equations, Ann. Scuola Norm. Sup. Pisa Cl. Sci. (4) 22 (2), 241-273 (1995).	

\bibitem{bmp} L.Boccardo, F. Murat and J.P. Puel, Existence of bounded solutions for nonlinear unilateral problems, Ann. Mat. Pura Appl. 152,  183-196 (1988).

\bibitem{bo} L. Boccardo and L. Orsina, Semilinear elliptic equations with singular nonlinearities, Calc. Var. and PDEs  37 (3-4), 363-380 (2010).

\bibitem{crt} M. G. Crandall, P. H. Rabinowitz and L. Tartar, On a Dirichlet problem with a singular nonlinearity, Comm.  Part. Diff. Eq. 2 (2), 193-222 (1977).

\bibitem{ddo}	L. M. De Cave, R. Durastanti and F. Oliva, Existence and uniqueness results for possibly singular nonlinear elliptic equations with measure data, NoDEA Nonlinear Differential Equations Appl. (2018) 25:18.

\bibitem{dos} F. Della Pietra, F. Oliva and S. Segura de Le\'on, Behaviour of solutions to p-Laplacian with Robin boundary conditions as p goes to 1, Proceedings of the Royal Society of Edinburgh Section A: Mathematics, in press, published online, doi:10.1017/prm.2022.92.


\bibitem{go} O. Guibé and A. Oropeza, Renormalized solutions of elliptic equations with Robin boundary conditions, Acta Math. Sci. Ser. B (Engl. Ed.) 37 (4), 889–910 (2017).

\bibitem{gmm} U. Guarnotta, S. Marano and D. Motreanu, On a singular Robin problem with convection terms, Adv. Nonlinear Stud. 20 (4), (2020) 895–909.

\bibitem{gmp} U. Guarnotta, S. Marano and N.S. Papageorgiou, Multiple nodal solutions to a Robin problem with sign-changing potential and locally defined reaction, Atti Accad. Naz. Lincei Rend. Lincei Mat. Appl. 30 (2), 269–294 (2019).

\bibitem{hunt} R. Hunt, On $L^{(p,q)}$ spaces, Enseign. Math. (2) 12, 249-276 (1966).

\bibitem{lm} A. C. Lazer and P. J. McKenna,  On a singular nonlinear elliptic boundary-value problem, Proc. Amer. Math. Soc.   111 (3), 721-730 (1991).

\bibitem{ll} J. Leray and J. L. Lions, Quelques r\'esulatats de Vi\text{$\check{s}$}ik sur les probl\'emes elliptiques nonlin\'eaires par les m\'ethodes de Minty-Browder, Bull. Soc. Math. France 93, 97-107 (1965).

\bibitem{mt} M. Montenegro and J.A.L. Tordecilla, Existence of positive solution for elliptic equations with singular terms and combined nonlinearities, J. Math. Anal. Appl. 503 (2), Paper No. 125316, 21 pp. (2021).

\bibitem{N}  J. Ne\v{c}as, Direct methods in the theory of elliptic equations Transl. from the French,
Springer Monographs in Mathematics, Berlin: Springer (2012).

\bibitem{op} F. Oliva and F. Petitta, Finite and Infinite energy solutions of singular elliptic problems: Existence and Uniqueness, Journal of Differential Equations 264, 311-340 (2018).


\bibitem{pri} A. Prignet, Non-homogeneous boundary value conditions for elliptic problems with measure valued right hand side, (Conditions aux limites non homogènes pour des problèmes elliptiques avec second membre mesure.)
Ann. Fac. Sci. Toulouse, VI. Sér., Math. 6 (2), 297-318 (1997).

\bibitem{vaz} J.L. Vázquez, A strong maximum principle for some quasilinear elliptic equations, Appl. Math. Optim. 12 (3), 191–202 (1984).
\end{thebibliography}
\end{document}